\documentclass[reqno]{amsart}
\usepackage{hyperref}
\usepackage{amssymb}

\allowdisplaybreaks

\theoremstyle{plain}
\newtheorem{theorem}{Theorem}[section]

\theoremstyle{definition}
\newtheorem{definition}[theorem]{Definition}

\theoremstyle{remark}
\newtheorem{remark}[theorem]{Remark}
\numberwithin{equation}{section}

\begin{document}

\title[Semilinear heat equation with memory boundary condition]
{Global existence of solutions of
semilinear heat equation with nonlinear memory condition}

\author[A. Gladkov]{Alexander Gladkov}
\address{Alexander Gladkov \\ Department of Mechanics and Mathematics
\\ Belarusian State University \\  4  Nezavisimosti Avenue \\ 220030
Minsk, Belarus  and  Peoples' Friendship University of Russia (RUDN University) \\  6 Miklukho-Maklaya street \\  117198 Moscow,  Russian Federation}    \email{gladkoval@mail.ru }

\author[M. Guedda]{Mohammed Guedda}
\address{Mohammed Guedda \\ Universit\'{e} de Picardie,
LAMFA, CNRS, UMR 6140, 33 rue Saint-Leu, F-80039,  Amiens, France}
\email{mohamed.guedda@u-picardie.fr}

\subjclass{35K20, 35K58, 35K61}
\keywords{Semilinear parabolic equation, memory boundary condition, finite time blow-up}

\begin{abstract}
We consider a semilinear parabolic equation with flux at the
boundary governed by a nonlinear memory. We give some conditions
for this problem which guarantee global existence of solutions as
well as blow up in finite time of all nontrivial solutions. The
results depend on the behavior of variable coefficients as $t \to
\infty.$
\end{abstract}

\maketitle

\section{Introduction}

We investigate the global solvability and blow-up in finite time for a 
semilinear heat equation with a nonlinear memory boundary condition:
\begin{equation}\label{1e}
u_{t} = \Delta u + c(t) u^p  \,\,\,  \textrm{for} \,\,\,
x \in \Omega, \,\,\, t>0,
\end{equation}
\begin{equation}\label{1b}
\frac{\partial u(x,t)}{\partial\nu} = k(t) \int_0^t u^q (x,\tau)
\,d\tau \,\,\,  \textrm{for} \,\,\, x  \in\partial\Omega, \,\,\,  t > 0,
\end{equation}
\begin{equation}\label{1i}
u(x,0)= u_0(x)  \,\,\,  \textrm{for}  \,\,\,  x  \in \Omega,
\end{equation}
where $\Omega$ is a bounded domain in $\mathbb{R}^n$ for $n \geq
1$ with smooth boundary $\partial  \Omega,$ $\nu$ is unit
outward normal on $\partial\Omega,$  $p >0$ and
$q>0.$ Here $c(t)$  and $k(t)$ are nonnegative continuous functions for $ t \geq  0.$
The initial datum $ u_0(x)$ is a nonnegative $C^1(\overline\Omega)$ function which satisfies the boundary condition at $t = 0.$

In the literature for parabolic equations, memory terms in the boundary flux appear in many references. For example, in \cite{C} a memory term (\ref{1b}) with $k(t) \equiv 1, \, q=1$ is introduced for the study of Newtonian radiation and calorimetry. A linear memory boundary condition takes into account the hereditary effects on the boundary as those studied in \cite{FM1}, \cite{GP}. In the paper \cite{FM2} similar hereditary boundary conditions have been employed in models of time-dependent electromagnetic fields at dissipative boundaries.
A nonlinear memory boundary condition arises  in a model of capillary growth in solid tumors as initiated by angiogenic growth factors, for example (see \cite{LPSN-H}).

Global existence and blow-up in finite time of solutions
for variety parabolic problems with memory boundary conditions
have been studied in many papers (see, for example, \cite{A,AD,ADD,ADW,DD,DW1,DW2,LQF,WCH} and the references therein).

Let $Q_T=\Omega\times(0,T),\;S_T =\partial\Omega\times(0,T),$ $\Gamma_T=S_T\cup\overline\Omega\times\{0\}$, $T>0.$
\begin{definition}\label{Def1}
We say that a nonnegative function
$u \in C^{2,1}(Q_T ) \cap C^{1,0}({Q}_T\cup \Gamma_T)$ is a subsolution of
problem (\ref{1e})--(\ref{1i})  in $Q_T$ if
\begin{eqnarray}\label{E:2.0}
\left\{ \begin{array}{ll} u_{t} \leq \Delta u + c(t) u^p  \,\,\,  \textrm{for} \,\,\,
(x,t) \in Q_T, \\
\frac{\partial u(x,t)}{\partial\nu} \leq  k(t) \int_0^t u^q (x,\tau)
\,d\tau \,\,\,  \textrm{for} \,\,\, (x,t) \in S_T, \\
u(x,0)\leq u_0(x) \,\,\, \textrm{for} \,\,\, x \in \Omega,
\end{array} \right.
\end{eqnarray}
and $u \in C^{2,1}(Q_T ) \cap C^{1,0}({Q}_T\cup \Gamma_T)$ is a supersolution
if $u \geq 0$ and it satisfies (\ref{E:2.0}) in the reverse order.
We say that $u$ is a solution of problem (\ref{1e})--(\ref{1i})  in $Q_T$
if it is both a subsolution and a supersolution of (\ref{1e})--(\ref{1i})
in $Q_T.$
\end{definition}

Local existence of solutions and comparison principle for (\ref{1e})--(\ref{1i}) may be developed using the same techniques as in \cite{ADD}, \cite{GladkovKavitova1}. We formulate comparison principle which will be used below.
\begin{theorem}\label{p:theorem:comp-prins}
     Let $u(x,t)$ and $v(x,t)$ be a supersolution and a subsolution
     of problem~(\ref{1e})--(\ref{1i}) in $Q_{T},$ respectively. Suppose that
     $u(x,t)>0$ or $v(x,t)>0$ in $Q_T\cup\Gamma_T$ if $\min(p,q)<1$.
     Then $u(x,t)\geq v(x,t)$ in $Q_T\cup\Gamma_T.$
\end{theorem}

In this paper we analyze the influence of variable coefficients on
global existence and blow-up in finite time of classical solutions of problem~(\ref{1e})--(\ref{1i}). Our global existence and blow-up results depend on the behavior of the functions $c (t)$ and $k (t)$ as $t \to \infty.$

This paper is organized as follows. In the next section we show
that all nonnegative solutions are global for $\max(p,q) \leq 1$ and present
finite time blow-up of all nontrivial solutions for $\max(p,q) > 1$ as well
as the existence of bounded global solutions for small initial data for $\min(p,q) > 1.$
In section~3 we investigate the case $ p=1, q > 1.$

\section{Finite time blow-up and global existence }\label{FT}

We begin with the global existence of solutions of  (\ref{1e})--(\ref{1i}).  The proof  relies on the continuation principle and the construction of a supersolution.
\begin{theorem}\label{Th00}
If $\max (p,q) \leq 1,$ then every solution of (\ref{1e})--(\ref{1i}) is global.
\end{theorem}
\begin{proof}
We seek a positive supersolution $\overline u$ of (\ref{1e})--(\ref{1i})  in
$Q_T$ for any positive $T.$ Since $c(t)$ and $k(t)$ are
continuous functions there exists a constant $M>0$ such that
$\max ( c(t), k(,t) )\leq M$ for $t \in [0,T].$  Let $\lambda_1$ be the
first eigenvalue of the following problem
\begin{equation*}
    \begin{cases}
        \Delta\varphi+\lambda\varphi=0,\;x\in\Omega,\\
        \varphi(x)=0,\;x\in\partial\Omega
    \end{cases}
\end{equation*}
and $\varphi(x)$ be the corresponding eigenfunction with
$\sup\limits_{\Omega}\varphi(x)=1$. It is well known
$\varphi(x)>0$ in $\Omega$ and $\max\limits_{\partial\Omega}
\partial\varphi(x)/\partial\nu < 0.$ We define
\begin{equation*}
\overline u = d \exp (bt) [2 - \varphi(x)],
\end{equation*}
where
\begin{equation*}
      b\geq \max \left( \lambda_1 + 2M , 2M \max\limits_{\partial\Omega}\left(-q\frac{\partial\varphi}{\partial\nu}\right)^{-1}\right),
\;d\geq \left\{ \sup\limits_\Omega u_0(x), 1 \right\}.
\end{equation*}
Then $\overline u$ satisfies
\begin{eqnarray*}\label{}
\begin{array}{ll}
\overline u_{t} \geq \Delta \overline u + c(t) \overline u^p  \,\,\, & \textrm{for} \,\,\,
(x,t) \in Q_T,   \\
\frac{\partial \overline u (x,t)}{\partial\nu} \geq k(t) \int_0^t \overline u^q (x,\tau)
\,d\tau \,\,\, & \textrm{for} \,\,\, (x,t) \in S_T, \\
\overline u(x,0) \geq u_0(x)  \,\,\, & \textrm{for} \,\,\,  x  \in \Omega.
\end{array}
\end{eqnarray*}
Hence, $\overline u$ is the desired supersolution and by Theorem~\ref{p:theorem:comp-prins} problem
(\ref{1e})--(\ref{1i})  has a global solution for any initial datum.
\end{proof}


We need the following assertion which was proved in \cite{GS} for a more general case.
\begin{theorem}\label{Th0}
Let $y (a) \geq 0,\,$ $y' (a) \geq 0,\,$ $y (a) + y' (a) > 0,\,$ $q> 1,\,$ $b (r)$ be a nonnegative continuous function for $r \geq a.$
Then for $r>a$ the inequality
\begin{equation*}\label{}
y''(r) \geq b(r) y^q (r)
\end{equation*}
has no global solutions if
\begin{equation*}\label{}
\int_a^\infty r^q b(r)  \,dr = \infty
\end{equation*}
and at least one of the following conditions is fulfilled
\begin{equation*}\label{}
 b(r)  \leq  \frac{B}{r^{q+1}} \,\,\,  \textrm{for large values of} \,\,\, r, \, B > 0,
\end{equation*}
or
\begin{equation*}\label{}
 b(r)  \,\,\,  \textrm{is nonincreasing for large values of} \,\,\, r.
\end{equation*}
\end{theorem}

Now we prove blow-up result for $\max (p,q) > 1.$
\begin{theorem}\label{Th1}
There are not nontrivial global solutions of (\ref{1e})--(\ref{1i}) if
\begin{equation}\label{1.2}
p >1 \,\,\,  \textrm{and} \,\,\,\int_0^\infty  c(t)  \,dt = \infty
\end{equation}
or
\begin{equation}\label{1.3}
q >1, \,\,\, k(t)  \geq \underline k(t) \geq 0  \,\,\,  \textrm{and} \,\,\,
\int_0^\infty t \underline k(t)  \,dt = \infty
\end{equation}
and at least one of the following conditions is fulfilled
\begin{equation}\label{1.4}
 \underline k(t)  \leq  \frac{c}{t^2} \,\,\,  \textrm{for large values of} \,\,\, t, \, c > 0,
\end{equation}
or
\begin{equation}\label{1.5}
 t^{1-q}\underline k(t)  \,\,\,  \textrm{is nonincreasing for large values of} \,\,\, t.
\end{equation}
\end{theorem}
\begin{proof}
Without loss of generality we can suppose that $u_0 (x) \not
\equiv 0$ in $\Omega.$ Then from strong maximum principle and (\ref{1b}) we conclude that $u (x,t) > 0$ for $x \in
\overline\Omega, \, t > 0$ and, moreover, by Theorem~\ref{p:theorem:comp-prins} we have $u (x,t) \geq \min_{\overline\Omega} u (x,t_0) >0$ for $x \in \overline\Omega, \, t \geq t_0$ and any $t_0 >0.$

Suppose at first that (\ref{1.2}) holds. Let us introduce an
auxiliary function
\begin{equation*}\label{1.6}
w(t)  = \int_\Omega u(x,t) \, dx.
\end{equation*}
Integrating (\ref{1e}) over $\Omega$ and using Green's identity, Jensen's inequality and boundary condition (\ref{1b}), we have
\begin{eqnarray}\label{1.7}
w'(t) &=& \int_\Omega \left( \Delta u(x,t) + c(t) u^p (x,t) \right)  \, dx = k(t)
\int_{\partial \Omega} \int_0^t u^q (x,\tau) d\tau \, dS \nonumber \\
&+& c(t)  \int_\Omega u^p (x,t)  \, dx \geq |\Omega|^{1-p} c(t) w^p (t).
\end{eqnarray}
From (\ref{1.2}) and (\ref{1.7}) we obtain blow-up of all nontrivial solutions.

Suppose now that either (\ref{1.3}), (\ref{1.4}) or (\ref{1.3}), (\ref{1.5})  hold.
Let $G (x,y;t-\tau)$ be the Green function of the heat equation with homogeneous Neumann
boundary condition. We note that $G (x,y;t-\tau)$ has the
following properties (see, for example, \cite{Hu_Yin}):
\begin{equation}\label{1.81}
        G (x,y;t-\tau) \geq 0, \; x,y \in\Omega, \; 0 \leq \tau <t,
    \end{equation}
 \begin{equation}\label{1.82}
        \int_{\partial\Omega}{G (x,y;t-\tau)}\,dS_x \geq c_1,\; y \in\partial\Omega, \; 0 \leq \tau < t.
    \end{equation}
Here and subsequently by $c_i\,(i\in \mathbb N)$ we denote
positive constants. It is well known that
problem~(\ref{1e})--(\ref{1i})  is equivalent to the equation
   \begin{eqnarray}\label{1.9}
        u(x,t) & = & \int_\Omega G (x,y;t) u_0(y) \,dy + \int_0^t \int_\Omega G (x,y;t-\tau) c(\tau) u^p(y,\tau) \,dy \, d\tau \nonumber \\
        &+& \int_0^t \int_{\partial\Omega} G (x,y;t-\tau) k(\tau) \int_0^\tau  u^q(y,\sigma) \, d\sigma  \,dS_y \,d\tau.
    \end{eqnarray}
Integrating  (\ref{1.9}) over $\partial\Omega$ and applying
(\ref{1.81}), (\ref{1.82}) and Jensen's inequality,  we obtain
\begin{eqnarray}\label{1.10}
\int_{\partial\Omega}  u(x,t)  \,dS_x & \geq &
c_1 \int_0^t  k(\tau) \int_0^\tau  \int_{\partial\Omega} u^q(y,\sigma) \, dS_y \,  d\sigma \,d\tau \nonumber \\
&\geq& c_1 |\partial\Omega|^{1-q} \int_0^t  k(\tau)
\tau^{1-q} \left( \int_0^\tau  \int_{\partial\Omega} u (y,\sigma)
dS_y d\sigma \right)^q  d\tau. \nonumber \\
    \end{eqnarray}
Let us define
 \begin{equation*}\label{1.10a}
   f(t) = \int_0^t \int_{\partial\Omega}  u(x,\sigma)  \,dS_x d\sigma.
    \end{equation*}
Then from (\ref{1.10}) we have
 \begin{equation}\label{1.11}
   f'(t) \geq  c_2 \int_0^t \tau^{1-q}  k(\tau)  f^q (\tau) d\tau .
    \end{equation}
After integration of (\ref{1.11}) over $[0,t]$ we obtain
 \begin{equation*}\label{}
   f (t) \geq  c_2 \int_0^t (t - \tau) \tau^{1-q}  k(\tau)  f^q (\tau) d\tau .
    \end{equation*}
Now we denote
 \begin{equation*}\label{}
g(t)  = c_2 \int_0^t (t - \tau) \tau^{1-q}  k(\tau)  f^q (\tau) d\tau .
    \end{equation*}
Then
 \begin{equation}\label{1.12}
g''(t)  = c_2  t^{1-q}  k(t)  f^q (t)  \geq c_2  t^{1-q}  \underline k(t) g^q (t) .
    \end{equation}
Applying Theorem~\ref{Th0} to (\ref{1.12}), we complete the proof.
\end{proof}

To formulate global existence result for problem (\ref{1e})--(\ref{1i})  we suppose that
\begin{equation}\label{1.13}
\int_0^\infty \left(  c(t)  + t k(t) \right) \,dt < \infty
\end{equation}
and there exist positive constants $\alpha,\;t_0$ and $K$ such that $\alpha>t_0$ and
\begin{equation}\label{1.14}
    \int_{t-t_0}^t {\frac{\tau k(\tau)}{\sqrt{t-\tau}}} \, d\tau \leq  K \, \textrm{ for } \, t \geq \alpha.
\end{equation}

\begin{theorem}\label{Th2}
Let $\min (p,q) >1$ and  (\ref{1.13}),  (\ref{1.14}) hold.
Then problem  (\ref{1e})--(\ref{1i}) has bounded global solutions for small initial data.
\end{theorem}
\begin{proof}
Let $y(x,t)$ be a solution of the following problem
\begin{equation*}\label{vsp}
\left\{
  \begin{array}{ll}
    y_t = \Delta y, \; x\in\Omega, \; t>0 \\
    \frac{\partial y(x,t)}{\partial \nu} = t k(t), \; x \in\partial\Omega, \; t>0, \\
    y(x,0)= 1,\; x\in\Omega.
  \end{array}
\right.
\end{equation*}
According to Lemma 3.3 of \cite{GladkovKavitova2} there exists a
positive constant $Y$ such that
\begin{equation*}\label{}
1 \leq y(x,t) \leq Y, \,  x\in\Omega, \; t>0.
\end{equation*}
Next, for any $T >0$ we construct a positive supersolution of
(\ref{1e})--(\ref{1i}) in $Q_T$ in such a form that
\begin{equation*}\label{}
\overline u (x,t) = \alpha z(t) y(x,t),
\end{equation*}
where $ \alpha >0$ and
\begin{equation*}\label{}
z (t) = \left( 1 + (p-1) (\alpha Y )^{p-1} \int_t^\infty c(\tau)
\, d\tau  \right)^{-\frac{1}{p-1}}.
\end{equation*}
It is easy to check that $z(t)$ is the solution of the equation
\begin{equation*}\label{}
z' (t) - (\alpha Y )^{p-1} c(t) z^p (t) = 0
\end{equation*}
and satisfies the inequality $ z (t) \leq 1.$ After simple
computations it follows that
\begin{eqnarray*}
        \overline u_t - \Delta \overline u - c(t) \overline u^p
& = & \alpha z' y + \alpha z y_t - \alpha z \Delta y - \alpha^p c(t) z^p y^p\\
        &\geq& \alpha y (z' - \alpha^{p-1} Y^{p-1} c (t) z^p) = 0, \;x\in\Omega, \; t>0,
    \end{eqnarray*}
and
    \begin{equation*}
\frac{\partial\overline u}{\partial\nu} - k(t) \int_0^t \overline
u^q (x,\tau) \, d\tau \geq \alpha t k(t) z(t) (1 - \alpha^{q-1}
Y^q) \geq 0, \; x\in\partial\Omega, \;t>0,
    \end{equation*}
if $\alpha \leq Y^{q/(q-1)}.$ Thus, by Theorem~\ref{p:theorem:comp-prins} there
exist bounded global solutions of (\ref{1e})--(\ref{1i}) for any initial data
satisfying the inequality
\begin{equation*}\label{}
u_0 (x) \leq  \alpha \left( 1 + (p-1) (\alpha Y )^{p-1}
\int_0^\infty c(\tau) \, d\tau  \right)^{-\frac{1}{p-1}}.
\end{equation*}
\end{proof}

 Let us introduce the following notations:
\begin{equation}\label{1.17}
 \ln_1 t= \ln t, \, \ln_{j+1} t=\ln(\ln_j t), \,
  l_j (t)= \prod_{i=1}^{j} \ln_i t, \,  l_{j,\gamma} (t) = l_j (t) \ln_j^{\gamma} t,   \, j \in \mathbb{N}, \, \gamma>0.
\end{equation}
\begin{remark}\label{Rem1}
 Arguing in the same way as in \cite{GladkovKavitova2} it is easy to show
that~(\ref{1.14}) is a necessary condition for the boundedness of
global solutions for~(\ref{1e})--(\ref{1i}).
It follows from  Theorem~\ref{Th1} and Theorem~\ref{Th2} that the condition  (\ref{1.2})  is optimal for blow-up in finite time of all nontrivial solutions of  (\ref{1e})--(\ref{1i}). Furthermore, from Theorem~\ref{p:theorem:comp-prins} and Theorem~\ref{Th1} we conclude that problem  (\ref{1e})--(\ref{1i}) has no nontrivial global solutions if $q > 1$ and
\begin{equation*}\label{}
k (t) \geq  \frac{c_3 }{t^2 l_j (t)} \,\,\,  \textrm{for} \,\,\, j \in \mathbb{N} \,\,\,  \textrm{and large values of} \,\,\, t.
\end{equation*}
On the other hand, from Theorem~\ref{Th2} we obtain
the existence of nontrivial bounded global solutions of (\ref{1e})--(\ref{1i}) if $\min (p,q) > 1,$
\begin{equation*}\label{}
\int_0^\infty c(t) \,dt < \infty \,\,\,  \textrm{and} \,\,\,
k (t) \leq  \frac{c_4 }{t^2 l_{j,\gamma} (t) } \,\,\,  \textrm{for} \,\,\, j \in \mathbb{N}, \gamma>0 \,\,\,  \textrm{and large values of} \,\,\, t.
\end{equation*}
\end{remark}


\section{Global existence and blow-up for $p=1, \, q > 1$ }\label{p=1, q>1}
In this section we obtain sufficient conditions for the
existence and nonexistence of global solutions of problem
(\ref{1e})--(\ref{1i})  for $p=1, q>1.$
\begin{theorem}\label{Th3} Let  $p=1, q>1.$Then there are not nontrivial global solutions of (\ref{1e})--(\ref{1i})
if $k(t)  \geq \underline k(t) \geq 0,$
\begin{equation}\label{2.2}
\int_0^\infty t^{1-q} \exp  \left( -\int_0^t c(s) \, ds \right)
\left( \int_0^t \exp  \left( \int_0^\tau c(s) \, ds \right) \,
d\tau  \right)^q  \underline k(t)   \,dt = \infty
\end{equation}
and at least one of the following conditions is fulfilled
\begin{equation}\label{2.3}
 t^{1-q} \exp \left( - 2 \int_0^t c(s)  ds \right)  \left( \int_0^t  \exp
 \left( \int_0^\tau c(s) \, ds  \right) \, d\tau  \right)^{q+1}  \underline k(t)   \leq  C,  \, C > 0,
\end{equation}
for large values of $t,$ or
\begin{equation}\label{2.4}
 t^{1-q}\exp \left( - 2\int_0^t c(s) \, ds \right)  \underline k(t) \,  \textrm{is nonincreasing for large values of} \,\,  t.
\end{equation}
\end{theorem}
\begin{proof}
We can suppose that $u_0 (x) \not \equiv 0,$  since otherwise $u
(x,t) \equiv 0.$ Let us change unknown function in the following
way
\begin{equation}\label{2.0}
u (x,t) = v (x,t) \exp \int_0^t c(\tau) \, d\tau.
\end{equation}
Then $ v (x,t) $ is a solution to the problem
\begin{eqnarray}\label{2.1}
\left\{
\begin{array}{ll}
v_{t} = \Delta v,  \, & x \in \Omega, \, t>0,   \\
\frac{\partial v(x,t)}{\partial\nu} = k(t) \int_0^t \exp
\left( q \int_0^\tau c(s) \, ds  - \int_0^t c(s) \, ds \right)  v^q (x,\tau) \,d\tau, \,  & x \in\partial\Omega, \, t > 0, \\
v(x,0)= u_0(x),  \, & \, x  \in \Omega.
\end{array}
\right.
\end{eqnarray}
It is well known that problem~(\ref{2.1}) is equivalent to the
equation
   \begin{eqnarray}\label{2.5}
       && v (x,t) =  \int_\Omega G (x,y;t) u_0 (y) \,dy \nonumber \\
        &+& \int_0^t \int_{\partial\Omega} G (x,y;t-\tau) k(\tau) \int_0^\tau
        \exp \left( q\int_0^\sigma c(s) ds - \int_0^\tau c(s) ds \right) v^q(y,\sigma) d\sigma dS_y d\tau. \nonumber \\
    \end{eqnarray}
Integrating  (\ref{2.5}) over $\partial\Omega$ and applying
(\ref{1.81}), (\ref{1.82}) and Jensen's inequality, we obtain
  \begin{eqnarray}\label{2.6}
&&\int_{\partial\Omega}  v(x,t)  \,dS_x \nonumber \\
        &\geq & c_1 |\partial\Omega|^{1-q} \int_0^t \tau^{1-q} k(\tau) \exp  \left(
        -\int_0^\tau c(s)  ds \right) \left( \int_0^\tau \int_{\partial\Omega}
        \exp  \left( \int_0^\sigma c(s) ds \right) v(y,\sigma) dS_y d\sigma \right)^q d\tau. \nonumber \\
    \end{eqnarray}
We set
 \begin{equation}\label{2.7}
   f(t) = \int_0^t  \int_{\partial\Omega}  \exp  \left( \int_0^\sigma c(s) ds \right)  v(y,\sigma)  \,dS_y d\sigma.
    \end{equation}
Then from (\ref{2.6}) and (\ref{2.7}) we deduce that $f'(t) >0$ for
$t >0$ and
\begin{equation}\label{2.8}
   f'(t) \geq  c_5  \exp  \left( \int_0^t c(s)  ds \right) \int_0^t \tau^{1-q}  \exp  \left(
        -\int_0^\tau c(s)  ds \right) k(\tau)  f^q (\tau) d\tau .
    \end{equation}
After integration of (\ref{2.8}) over $[0,t]$ we obtain
 \begin{equation*}\label{}
   f (t) \geq  c_5 \int_0^t \exp  \left( \int_0^\sigma c(s)  ds \right) \int_0^\sigma \tau^{1-q}  \exp  \left(
        -\int_0^\tau c(s)  ds \right) k(\tau)  f^q (\tau) d\tau
        d\sigma.
    \end{equation*}
Defining
 \begin{equation*}\label{}
g(t)  = c_5 \int_0^t \exp  \left( \int_0^\sigma c(s)  ds \right)
\int_0^\sigma \tau^{1-q}  \exp  \left(
        -\int_0^\tau c(s)  ds \right) k(\tau)  f^q (\tau) d\tau
        d\sigma
    \end{equation*}
we have $f(t) \geq g(t).$ Moreover,
 \begin{equation}\label{2.9}
g''(t) \geq c(t) g' (t) + c_5  t^{1-q} \underline k(t) g^q (t).
    \end{equation}
Multiplying (\ref{2.9})  by $ \exp  \left( - \int_0^t c(s)  ds \right),$ we obtain
 \begin{equation}\label{2.10}
\left(  \exp  \left( - \int_0^t c(s)  ds \right) g'(t) \right)'
\geq c_5  t^{1-q} \exp  \left( - \int_0^t c(s)  ds \right)
\underline k(t)  g^q (t).
    \end{equation}
Let us change variable and unknown function in the following way
 \begin{equation*}\label{}
s =  \int_0^t \exp  \left( \int_0^\tau c(\sigma)  d\sigma \right) \, d\tau, \,\,\,
\phi (s) = g (t)
    \end{equation*}
and rewrite (\ref{2.10}) as
 \begin{equation}\label{2.11}
\phi'' (s)  \geq c_5 t^{1-q} \underline k(t) \exp \left( -2
\int_0^t c(\sigma)  d\sigma \right)  \phi^q (s).
    \end{equation}
Applying Theorem~\ref{Th0} to (\ref{2.11}), we complete the proof.
\end{proof}

\begin{theorem}\label{Th4} Let $p =1,\, q >1,$
\begin{equation}\label{2.12}
\int_0^\infty  k(t) \exp  \left( -\int_0^t c(s) \, ds \right)
\int_0^t \exp  \left( q \int_0^\tau c(s) \, ds \right) \, d\tau
\,dt < \infty
\end{equation}
and there exist positive constants $\alpha,\;t_0$ and $K$ such
that $\alpha > t_0$ and
\begin{equation}\label{2.13}
    \int_{t-t_0}^t {\frac{k(\tau) \exp  \left( -\int_0^\tau c(s) \, ds \right)
\int_0^\tau \exp \left( q \int_0^\sigma c(s) \, ds \right) \,
d\sigma \,d\tau }{\sqrt{t-\tau}}} \leq  K \, \textrm{ for } \, t
\geq \alpha.
\end{equation}
 Then problem  (\ref{1e})--(\ref{1i}) has global solutions for small initial data.
 If, in addition,
 \begin{equation}\label{2.14}
  \int_0^\infty c(t) \, dt < \infty
 \end{equation}
then problem  (\ref{1e})--(\ref{1i}) has bounded global solutions for small initial data.
\end{theorem}
\begin{proof}
To prove the theorem we construct a positive supersolution of
(\ref{1e})--(\ref{1i}) in such a form that
\begin{equation*}\label{}
\overline u (x,t) = \alpha \exp  \left( \int_0^t c(s) \, ds
\right) h(x,t),
\end{equation*}
where $h(x,t)$ is a solution to the following problem
\begin{equation}\label{vsp2}
\left\{
  \begin{array}{ll}
    h_t = \Delta h, \; x\in\Omega, \; t>0, \\
    \frac{\partial h(x,t)}{\partial \nu} =  k(t) \exp  \left( -\int_0^t c(s) \, ds \right)
\int_0^t \exp  \left( q \int_0^\tau c(s) \, ds \right) \, d\tau
    , \; x \in\partial\Omega, \; t>0, \\
    h(x,0)= 1,\; x\in\Omega.
  \end{array}
\right.
\end{equation}
As it is proved in \cite{GladkovKavitova2} the solution of
(\ref{vsp2}) satisfies the inequalities
\begin{equation*}\label{}
1 \leq h(x,t) \leq H, \,  x\in\Omega, \; t>0
\end{equation*}
for some $H>0.$ It is easy to check that $\overline u (x,t)$ is
the supersolution of (\ref{1e})--(\ref{1i}) if $\alpha \leq H^{-q/(q-1)}$ and
$u_0 (x) \leq  \alpha.$ Moreover, $\overline u (x,t)$ is bounded function under the condition (\ref{2.14}).
\end{proof}

\begin{remark}\label{Rem2}
Let $p=1, q>1.$  Arguing in the same way as in \cite{GladkovKavitova2} it is easy to prove from (\ref{2.0}) and (\ref{2.1})
that both conditions (\ref{2.13}) and (\ref{2.14}) are necessary for the boundedness of
global solutions of~(\ref{1e})--(\ref{1i}). Furthermore, we conclude from Theorem~\ref{p:theorem:comp-prins},
Theorem~\ref{Th3} and Theorem~\ref{Th4}  that  problem (\ref{1e})--(\ref{1i})
has no nontrivial global solutions if
\begin{equation*}\label{}
c (t) \geq \frac{\beta }{t} \,\,\, \textrm{for large values of}
\,\,\, t \,\,\, \textrm{and some} \,\,\, \beta > 0
\end{equation*}
and
\begin{equation*}\label{}
k (t) \geq  \frac{c_6 }{t^{\beta (q-1) + 2} l_j (t)} \,\,\,
\textrm{for} \,\,\, j \in \mathbb{N} \,\,\,  \textrm{and large
values of} \,\,\, t
\end{equation*}
and problem (\ref{1e})--(\ref{1i}) has nontrivial bounded global solutions if
\begin{equation*}\label{}
c (t) \leq \frac{\omega }{t} \,\,\, \textrm{for large values of}
\,\,\, t \,\,\, \textrm{and some} \,\,\, \omega > 0
\end{equation*}
 and
\begin{equation*}\label{}
k (t) \leq  \frac{c_7}{t^{\omega (q-1) + 2} l_{j,\gamma} (t)}
\,\,\, \textrm{for} \,\,\, j \in \mathbb{N}, \gamma >0 \,\,\,
\textrm{and large values of} \,\,\, t,
\end{equation*}
where $l_j (t)$ and $l_{j,\gamma} (t)$ were introduced in (\ref{1.17}).
\end{remark}

\section*{Acknowledgement}

The first author is grateful to the University of Picardie Jules Verne,  since a part of this work was done while he enjoyed the hospitality of this  university.

\section*{Funding}

The first author was supported by the "RUDN University Program 5-100" and
the state program of fundamental research of Belarus (grant 1.2.03.1). The second author was supported by  DAI-UPJV F-Amiens.


\begin{thebibliography}{9999}

  \bibitem{C}  Ciarletta M. A differential problem for heat equation with a boundary condition with memory. Appl. Math. Letters. 1997;10:95--101.

 \bibitem{FM1} Fabrizio M, Morro A. Mathematical problems in linear viscoelasticity. Vol. 12, SIAM studies in applied mathematics. Philadelphia (PA): SIAM; 1992.

\bibitem{GP} Gurtin ME, Pipkin AP. A general theory of heat conduction with finite wave speeds. Arch. Rational Mech. Anal. 1968;31:113--126.

 \bibitem{FM2} Fabrizio M, Morro A.  A boundary condition with memory in electromagnetism, Arch. Rational Mech. Anal.
 1996;136:359--381.

\bibitem{LPSN-H} Levine HA, Pamuk S, Sleeman B, Nilsen-Hamilton M. Mathematical modeling of capillary
formation and development in tumor angiogenesis: penetration into the stroma. Bull. Math. Biol. 2001; 63: 801--863.

\bibitem{A} Anderson J. A fast diﬀusion model with memory at the boundary: global
solvability in the critical case. Appl. Anal. 2017; 96(5):771--777.

\bibitem{AD} Anderson J, Deng K. Global solvability for the porous medium equation with boundary flux governed by nonlinear memory. J. Math. Anal. Appl. 2015;423:1183--1202.

\bibitem{ADD} Anderson J, Deng K, Dong Z. Global solvability for the heat equation with boundary flux governed by nonlinear memory. Quart. Appl. Math. 2011;69:759--770.

\bibitem{ADW} Anderson J, Deng K, Wang Q. Global behavior of solutions to the fast diﬀusion equation with boundary boundary flux governed by memory. Math. Methods Appl. Sci. 2016;39(15):4451--4462.

\bibitem{DD} Deng K, Dong Z.  Blow-up for the heat equation with a general memory boundary condition. Commun. Pure App. Anal.  2012;11:2147--2156.

\bibitem{DW1} Deng K, Wang Q.  Blow-up rate for the heat equation with a memory boundary condition. Appl. Anal. 2015;94(2): 308--317.

\bibitem{DW2} Deng K, Wang Q. Global existence and blow-up for the fast diffusion equation with a memory boundary condition. Q. Appl. Math. 2016;74(1): 189--199.

\bibitem{LQF} Li C, Qiu L, Fang ZB. General decay rate estimates for a semilinear parabolic equation with memory term and mixed boundary condition. Bound. Value Probl. 2014;197:11p.

\bibitem{WCH} Wang Y, Chen J, He C. Singularity analysis for a semilinear
integro-differential equation with nonlinear memory boundary. J. Inequal. Appl. 2014;472:17p.


\bibitem{GladkovKavitova1} Gladkov A, Kavitova T.
 Initial boundary value problem for a semilinear parabolic equation with nonlinear nonlocal boundary conditions. Ukrain. Math. J. 2016;68(2):179--192.

\bibitem{GS} Gladkov A, Slepchenkov N. On proper and entire solutions
  of a generalized Emden--Fowler equation. Differ. Equat. 2005;41(2):173--183.

\bibitem{Hu_Yin} Hu B, Yin HM. Critical exponents for a system of heat equations coupled in a non-linear boundary condition. Math. Meth. Appl. Sci. 1996;19(14):1099--1120.

\bibitem{GladkovKavitova2} Gladkov A, Kavitova T.
Blow-up problem for semilinear heat equation with nonlinear
nonlocal boundary condition. Appl. Anal. 2016;95(9):1974--1988.

\end{thebibliography}
\end{document}